\pdfoutput=1
\documentclass[a4paper]{amsart}
\usepackage{fixltx2e}
\usepackage[utf8]{inputenc}
\usepackage[T1]{fontenc}
\usepackage[charter]{mathdesign}
\usepackage[bookmarks=true,pdfdisplaydoctitle=true]{hyperref}
\hypersetup{
pdfpagemode=UseNone,
pdfstartview=FitH,
pdftitle={Sparse domination of uncentered variational operators},
pdfauthor={Fernanda Clara de França Silva and Pavel Zorin-Kranich},
pdflang=en-US
}

\usepackage[safeinputenc=true,style=alphabetic,firstinits,maxnames=5,doi=false,eprint=true,isbn=false,url=false]{biblatex}

\newbibmacro{string+doiurl}[1]{%
  \iffieldundef{doi}{%
    \iffieldundef{url}{%
         #1%
      }{%
      \href{\thefield{url}}{#1}%
    }%
  }{%
    \href{http://dx.doi.org/\thefield{doi}}{#1}%
  }%
}
\DeclareFieldFormat[article,misc]{title}{\usebibmacro{string+doiurl}{\mkbibquote{#1}}}

\usepackage{esint}

\newcommand*{\mailto}[1]{\href{mailto:#1}{#1}}

\numberwithin{equation}{section}
\newtheorem{theorem}[equation]{Theorem}
\newtheorem{lemma}[equation]{Lemma}

\theoremstyle{definition}
\newtheorem{definition}[equation]{Definition}
\theoremstyle{remark}

\newtheorem{example}[equation]{Example}

\newcommand*{\N}{\mathbb{N}}
\newcommand*{\Z}{\mathbb{Z}}
\newcommand*{\R}{\mathbb{R}}
\newcommand*{\C}{\mathbb{C}}
\newcommand*{\cD}{\mathcal{D}}
\newcommand*{\cP}{\mathcal{P}}
\newcommand*{\cQ}{\mathcal{Q}}
\newcommand*{\cS}{\mathcal{S}}
\newcommand*{\cX}{\mathcal{X}}
\newcommand*{\dif}{\mathrm{d}}
\newcommand*{\dist}{\mathrm{dist}}

\def\<{\left\langle}
\def\>{\right\rangle}
\newcommand*{\widevec}[1]{\overrightarrow{#1}}

\newcommand{\Vop}[1][r]{\mathcal{V}^{#1}}
\newcommand{\hVop}[1][r]{\dot{\mathcal{V}}^{#1}}

\newcommand{\Nop}{\mathcal{N}}

\newcommand{\hV}[3][r]{\dot V^{#1}(#2 : #3)}
\newcommand{\Aop}{\mathcal{A}}
\newcommand{\Top}{\mathcal{T}}

\newcommand{\Dini}{\mathrm{Dini}}

\newcommand{\calS}{\mathcal{S}}
\newcommand{\calD}{\mathcal{D}}
\newcommand{\calQ}{\mathcal{Q}}
\newcommand{\mean}[2][Q]{\<#2\>_{#1}}
\def\clap#1{\hbox to 0pt{\hss#1\hss}}

\def\mathclap{\mathpalette\mathclapinternal}

\def\mathclapinternal#1#2{%
\clap{$\mathsurround=0pt#1{#2}$}}
\newcommand{\clapint}[1]{\int\limits_{\mathclap{#1}}}

\begin{document}
\allowdisplaybreaks
\subjclass[2010]{42B20 (Primary) 42B25 (Secondary)}
\title{Sparse domination of sharp variational truncations}
\author{Fernanda Clara de França Silva}
\address{Universität Tübingen\\
  Mathematisches Institut\\
  Auf der Morgenstelle 10\\
  72076 Tübingen\\
  Germany 
}
\email{\mailto{fesi@fa.uni-tuebingen.de}}

\author{Pavel Zorin-Kranich}
\address{Universit\"at Bonn\\
  Mathematisches Institut\\
  Endenicher Allee 60\\
  53115 Bonn\\
  Germany
}
\email{\mailto{pzorin@uni-bonn.de}}
\urladdr{\url{http://www.math.uni-bonn.de/people/pzorin/}}
\begin{abstract}
We provide a versatile formulation of Lacey's recent sparse pointwise domination technique with a local weak type estimate on a nontangential maximal function as the only hypothesis.
We verify this hypothesis for sharp variational truncations of singular integrals in the case when unweighted $L^{2}$ estimates are available.
This extends previously known sharp weighted estimates for smooth variational truncations to the case of sharp variational truncations.
We also include a sparse domination result for iterated commutators of multilinear operators with BMO functions.
\end{abstract}
\thanks{First author partially supported by a CAPES/DAAD scholarship.}
\maketitle

\section{Introduction}
\subsection{Sparse domination}
Sparse domination has been introduced by Lerner \cite{MR3085756} in order to simplify the proof of Hyt\"onen's $A_{2}$ theorem (see \cite{MR3204859} for a comprehensive history of this result).
A new approach to sparse domination via weak type endpoint estimates has been recently discovered by Lacey \cite[Theorem 4.2]{arXiv:1501.05818}, quantitatively refined by Hytönen, Roncal, and Tapiola \cite[Theorem 2.4]{arXiv:1510.05789}, and streamlined by Lerner \cite{arXiv:1512.07247}.
Our first result is an abstract implementation of Lacey's argument that can be applied as a black box in a number of situations, for instance to multilinear operators (recovering the sparse domination result in \cite{arXiv:1512.02400}), to intrinsic square functions (see \cite{arXiv:1605.02936}, where the second author uses Theorem~\ref{thm:sparse-domination} to extend some results in \cite{arXiv:1505.00195}), and also to variational truncations of singular integrals that will be the second topic of this article.

We will use the following version of the nontangential maximal function.
Let $(X,\rho,\mu)$ be a space of homogeneous type (see Section~\ref{sec:prelim} for definitions) and let $F$ be a function on the set
\[
\cX := \{(x,s,t)\in X\times (0,\infty)\times (0,\infty) : s\leq t\}.
\]
We define the non-tangentional maximal operator (of aperture $a\geq 0$) localized to a set $Q\subset X$ by
\[
(\Nop_{a,Q} F)(x) := 1_{Q}(x) \sup_{y\in X,\rho(x,y)<as<at\leq \dist(y,X\setminus Q)} F(y,s,t).
\]
We will omit $Q$ from the notation if $Q=X$ and we will also omit $a$ if $a=1$.
Now we state our version of Lacey's sparse domination principle.
The notions of adjacent systems of dyadic cubes and sparse collections are recalled in Section~\ref{sec:prelim}.
\begin{theorem}
\label{thm:sparse-domination}
For every space of homogeneous type $(X,\rho,\mu)$ and every choice of adjacent systems of dyadic cubes $\cD^{\alpha}$ there exist $\epsilon,\eta>0$ such that the following holds.
Let $F:\cX\to [0,\infty]$ be a function that is monotonic in the sense that
\[
s\leq s'\leq t'\leq t \implies F(x,s',t')\leq F(x,s,t)
\]
and subadditive in the sense that
\[
s\leq s'\leq t \implies F(x,s,t)\leq F(x,s,s')+F(x,s',t).
\]
Suppose that for every dyadic cube $Q$ there exists $c_{Q}\geq 0$ such that
\begin{equation}
\label{eq:local-weak(1,1)-bound}
\mu\{\Nop_{Q}F > c_{Q}\} \leq \epsilon \mu(Q).
\end{equation}
Then there exist $\eta$-sparse collections $\mathcal{S}^{\alpha,k_{0}} \subset \mathcal{D}^{\alpha}$ of cubes such that
\begin{equation}
\label{eq:sparse-domination}
\Nop F \leq \liminf_{k_{0}\to -\infty} \sum_{\alpha} \sum_{Q\in\mathcal{S}^{\alpha,k_{0}}} 1_{Q} c_{Q}
\end{equation}
holds pointwise almost everywhere.
\end{theorem}
One situation in which Theorem~\ref{thm:sparse-domination} does not apply as a black box is that of commutators of (multi)linear operators with BMO functions, and we provide the necessary modifications to the argument in Section~\ref{sec:commutator}, where a multilinear extension of \cite[Theorem 1.1]{arXiv:1604.01334} is proved.

\subsection{Variational truncations of singular integrals}
In this part of the article we return to the space $X=\R^{d}$ with the Euclidean distance and the Lebesgue measure.
Let $K$ be an $\omega$-Calder\'on--Zygmund (CZ) kernel (see Section~\ref{sec:prelim} for definitions) and consider the corresponding truncation operator given by
\begin{equation}
\label{eq:Top}
\Top f(x,s,t) := \int_{s<|x-y|<t} K(x,y) f(y) \dif y.
\end{equation}
For $1\leq r<\infty$ we define the homogeneous%
\footnote{The dot in the notation ``$\dot{V}^{r}$'' is not standard and is motivated by the embeddings $\dot{B}_{r}^{1/r,1} \to \dot{V}^{r} \to \dot{B}_{r}^{1/r,\infty}$ between the spaces of bounded homogeneous variation and homogeneous Besov spaces \cite{MR0380389}.}
variation operator, acting on functions on $\cX$, by
\[
(\hVop F)(x,s,t) := \sup_{s\leq t_{1}<\dots <t_{J} \leq t} \big( \sum_{j=1}^{J-1} |F(x,t_{j},t_{j+1})|^{r} \big)^{1/r},
\]
and analogously for $r=\infty$ with the $\ell^{\infty}$ norm in place of the $\ell^{r}$ norm.

It is known that, if the kernel $K$ is of convolution type, i.e.\ $K(x,y)=k(x-y)$, satisfies the cancellation condition
\[
\int_{\partial B(0,t)} k(x) \dif x = 0,
\qquad t>0,
\]
and satisfies one of the following additional conditions:
\begin{enumerate}
\item\label{lem:VT-unweighted:homogeneous} the kernel $k$ is homogeneous of degree $-d$, that is, $k(tx)=t^{-d}k(x)$ for $t>0$, or
\item\label{lem:VT-unweighted:smooth} the kernel $k$ satisfies the smoothness condition $|k'(y)| \lesssim |y|^{-d-1}$,
\end{enumerate}
then, for $r>2$, the operator $\Nop_{0} \circ \hVop \circ \Top$ is bounded on $L^{p}(\R^{d})$ and has weak type (1,1).
The strong type bounds in the case \ref{lem:VT-unweighted:homogeneous} have been proved in \cite[Theorem A]{MR1953540} (see also \cite{MR2434308} and \cite{arXiv:1508.03872}) and in the case \ref{lem:VT-unweighted:smooth} in \cite[Theorem A.1]{arXiv:1512.07523}.
In both cases the $L^{p}$ bounds imply the weak type $(1,1)$ bound by \cite[Theorem B]{MR1953540} (note that the assumption (1.8) in that article follows from the Dini condition).

Our second main result is that these bounds remain true with $\Nop_{0}$ replaced by $\Nop_{a}$, $a>0$.
\begin{theorem}
\label{thm:nontangentional-variation}
Let $K$ be an $\omega$-CZ kernel on $\R^{d}$, $r>2$, and assume that $\Nop_{0} \circ \hVop \circ \Top$ has weak type (1,1).
Then also $\Nop_{a} \circ \hVop \circ \Top$ has weak type (1,1) for every $a>0$.
\end{theorem}
The novelty of this result are the sharp truncations in \eqref{eq:Top}.
An analogous result with $1_{(s,t)}$ replaced by appropriately scaled smooth truncations is implicitly contained in \cite{MR3065022}.

The appearance of cones with positive aperture in Theorem~\ref{thm:nontangentional-variation} allows us to apply Theorem~\ref{thm:sparse-domination} to variational truncations of singular integrals.
Indeed, the localized operator $\Nop_{Q}\circ\hVop\circ\Top$ is dominated by the global operator $\Nop\circ\hVop\circ\Top$, and therefore has weak type (1,1) uniformly in $Q$.
On the other hand, the localized operator depends only on the values of $f$ on $Q$, and therefore \eqref{eq:local-weak(1,1)-bound} is satisfied for the function $F=\hVop\Top f$ with $c_{Q}=\frac{C}{\epsilon} |Q|^{-1} \int_{Q} |f|$.
Therefore, $\Nop \circ \hVop \circ \Top f$ can be estimated by sparse operators of the form \eqref{eq:sparse-operator}.

Sparse operators are known to satisfy very good weighted estimates, the currently best results can be found in \cite{arXiv:1509.00273} ($L^{p}$ bounds with $p>1$) and \cite{MR3455749} (the weak type (1,1) endpoint).
Consequently, we obtain sharp weighted estimates for the variationally truncated operators $\Nop \circ \hVop \circ \Top$, unifying the previous results for sharp truncations with unspecified dependence on the characteristic of the weight \cite{MR3283159,arXiv:1511.05129} and for smooth truncations with sharp dependence on the characteristic of the weight \cite{MR3065022}.

\section{Notation and preliminaries}
\label{sec:prelim}
\subsection{Spaces of homogeneous type}
A \emph{quasi-metric} on a set $X$ is a function $\rho : X\times X\to [0,\infty)$ such that $\rho(x,y)=0 \iff x=y$ that is symmetric and satisfies the quasi-triangle inequality
\[
\rho(x,y) \leq A_{0} (\rho(x,z) + \rho(z,y))
\quad\text{for all}\quad
x,y,z\in X
\]
with some $A_{0}<\infty$ independent of $x,y,z$.
A measure $\mu$ on a quasi-metric space $(X,\rho)$ is called \emph{doubling} if there exists $A_{1}<\infty$ such that
\[
\mu(B(x,2r)) \leq A_{1} \mu(B(x,r))
\quad\text{for all}\quad
x\in X,r>0,
\]
where $B(x,r)=\{y\in X : \rho(x,y)<r\}$ are the quasimetric balls of radius $r$.
These balls need not be open, but can be made open upon passing to an equivalent quasi-metric \cite{MR546295}.
A tuple $(X,\rho,\mu)$ consisting of a set $X$, a quasi-metric $\rho$, and a doubling measure $\mu$ is called a \emph{space of homogeneous type}.
We will frequently denote the measure of a set $Q$ by $|Q|=\mu(Q)$ and the average of a function $f$ over $Q$ by $\mean{f}=|Q|^{-1} \int_{Q} f \dif\mu$.

\subsection{Adjacent systems of dyadic cubes}
Filtrations on spaces of homogeneous type that closely resemble dyadic filtrations on $\R^{d}$ have been first constructed by Christ \cite{MR1096400} and are now commonly known as \emph{Christ cubes}.
For our purposes we do not need the small boundary property enjoyed by the Christ cubes, but we do need adjacent systems of cubes that have covering properties similar to those of shifted dyadic cubes in $\R^{d}$.
Such systems have been constructed in \cite{MR2901199}.
\begin{definition}
Let $(X,\mu)$ be a measure space.
A \emph{system of dyadic sets} $\cD$ consists of a sequence $(\cD_{k})_{k\in\Z}$ of collections of measurable subsets of $X$ such that for all $l\leq k$, $l,k\in\Z$
\begin{enumerate}
\item\label{cc:cover} $\mu(X\setminus\cup_{Q\in\cD_{k}}Q)=0$.
\item\label{cc:nested} For each $Q\in\cD_{k}$ and $Q'\in\cD_{l}$ either $Q\subseteq Q'$ or $Q\cap Q'=\emptyset$.
\end{enumerate}
\end{definition}
By an abuse of notation the sets $Q$ remember their generation $k(Q)$ (the \emph{unique} number such that $Q\in\cD_{k(Q)}$), even though it is allowed that the same $Q$ (viewed as a set) may occur in different generations $\cD_{k}$.
The relation $Q'\subseteq Q$ includes the inequality $k(Q')\geq k(Q)$ and the relation $Q'=Q$ includes $k(Q')=k(Q)$.
\begin{definition}
\label{def:cc}
Let $(X,\rho,\mu)$ be a quasi-metric measure space and assume that the measure $\mu$ has full support.
A \emph{system of dyadic cubes} is a system of dyadic sets $\cD$ such that for some $0<\delta<1$, $0<c_{1} \leq C_{1}<\infty$ and all $k\in\Z$ and $Q=Q_{\alpha}^{k}\in\cD_{k}$ there exists $z=z(Q)=z_{\alpha}^{k}\in X$ such that $B(z,a_{0}\delta^{k}) \subseteq Q \subseteq B(z,C_{1}\delta^{k})$.
\end{definition}

\begin{definition}
Let $(X,\rho,\mu)$ be a quasi-metric measure space and assume that the measure $\mu$ has full support.
Systems of dyadic cubes $\cD^{\alpha}$, $\alpha\in A$, are said to be \emph{adjacent} if there exists $C_{3}<\infty$ such that for every $z\in X$ and $r>0$ there exist $\alpha\in A$, $k\in\Z$, and $Q\in\cD^{\alpha}_{k}$ such that $B(z,r) \subset Q \subset B(z,C_{3}r)$.
\end{definition}
It is known that in every space of homogeneous type there exists a finite collection of adjacent systems systems of dyadic cubes \cite[Theorem 4.1]{MR2901199}.

\begin{example}
Let $X=\R^{d}$ with the Euclidean distance and the Lebesgue measure.
For each $\alpha\in\{0,1,2\}^{d}$ the corresponding \emph{shifted system of dyadic cubes} is given by
\[
\cD^{\alpha} = \{ 2^{-k}([0,1)^{d} + m + (-1)^{k}\frac13 \alpha), k\in\Z, m\in\Z^{d}\}.
\]
Then the systems $\mathcal{D}^{\alpha}$, $\alpha\in\{0,1,2\}^{d}$, are adjacent. 
In fact, on $\R^{d}$ one can construct $d+1$ shifted systems of dyadic cubes that are adjacent \cite{MR1993970}. 
\end{example}

\begin{example}
Let $(X,\mu)$ be a measure space and let $\cD$ be a system of dyadic sets.
Define a metric on $X$ by
\[
\rho(x,x') := \inf\{ 2^{-k} : \exists Q\in\cD_{k} : x,x'\in Q\}.
\]
Then the system $\cD$ is a system of dyadic cubes with respect to this metric, and this system is adjacent.
For instance, the standard dyadic cubes in $\R^{d}$ are an adjacent system of dyadic cubes with respect to the dyadic metric.
This does not preclude one from considering CZ operators on $\R^{d}$ with respect to the Euclidean metric and allows one to recover Lerner's version \cite{arXiv:1512.07247} of the pointwise sparse domination theorem from Theorem~\ref{thm:sparse-domination}.
\end{example}

\subsection{Sparse and Carleson collections}
Let $\cD$ be a system of dyadic sets on a measure space $(X,\mu)$.
A collection $\cS\subset\cD$ is called
\begin{enumerate}
\item \emph{$\eta$-sparse} if there exist pairwise disjoint subsets $E(Q)\subset Q\in\mathcal{S}$ with $|E(Q)| \geq \eta |Q|$ and
\item \emph{$\Lambda$-Carleson} if one has $\sum_{Q'\subset Q, Q'\in\cS} \mu(Q') \leq \Lambda \mu(Q)$ for all $Q\in\cD$.
\end{enumerate}
It is known that a collection is $\eta$-sparse if and only if it is $1/\eta$-Carleson \cite[\textsection 6.1]{arXiv:1508.05639}.
The corresponding \emph{sparse operator} is given by
\begin{equation}
\label{eq:sparse-operator}
A_{\mathcal{S}}f = \sum_{Q\in\mathcal{S}} 1_{Q} \<f\>_{Q},
\quad\text{where}\quad
\<f\>_{Q} = |Q|^{-1} \int_{Q} f.
\end{equation}

\subsection{$\omega$-Calder\'on--Zygmund kernels}
An $\omega$-Calder\'on--Zygmund (CZ) kernel is a function $K:\R^{d}\times \R^{d} \setminus (\mathrm{diagonal}) \to \C$ that satisfies the size estimate
\begin{equation}
\label{eq:size}
|K(x,y)| \leq \frac{C_{K}}{|x-y|^{d}}
\end{equation}
and the smoothness estimate
\begin{equation}
\label{eq:smoothness}
|K(x,y)-K(x',y)| + |K(y,x)-K(y,x')| \leq \omega\big( \frac{|x-x'|}{|x-y|} \big) \frac{1}{|x-y|^{d}}
\end{equation}
for $|x-y|>2|x-x'|>0$ with some \emph{modulus of continuity} $\omega :[0,\infty) \to [0,\infty)$ (that is, a subadditive function: $\omega(t+s)\leq \omega(t)+\omega(s)$ for all $s,t\geq 0$) that satisfies the \emph{Dini condition}
\begin{equation}
\label{eq:Dini}
\|\omega\|_{\Dini} := \int_{0}^{1} \omega(t) \frac{\dif t}{t} < \infty.
\end{equation}

\section{Uncentered variational estimates}
\label{sec:no-weights}
Consider the averaging operator
\begin{equation}
\label{eq:Aop}
\Aop f(x,s,t) := A_{t}f(x)-A_{s}f(x),
\quad
A_{t}f(x) := \fint_{|x-y|<t} f(x+y) \dif y.
\end{equation}
It satisfies the following uncentered variational estimates.
\begin{lemma}
\label{lem:VA-unweighted}
Let $r>2$ and $a\geq 0$.
Then $\Nop_{a}\circ \hVop\circ\Aop$ is bounded on $L^{p}(\R^{d})$, $1<p<\infty$, and has weak type $(1,1)$.
\end{lemma}
Variational estimates for averaging operators go back to \cite[Section 3]{MR1019960}.
The only new aspect of Lemma~\ref{lem:VA-unweighted} is that the variations are maximized over a conical region when $a>0$.
This is easy to achieve using the uncentered square function from \cite{arxiv:1409.7120}.
\begin{proof}[Sketch of proof]
The $L^{p}$, $1<p<\infty$, bound for the dyadic version of this operator is a direct consequence of L\'epingle's inequality for martingales.
The real version can be compared with the dyadic version using the uncentered square function from \cite[Theorem 1.4]{arxiv:1409.7120}.
Finally, the weak type $(1,1)$ bound follows by \cite[Proposition 5.1]{arxiv:1409.7120}.

Note that the results cited from \cite{arxiv:1409.7120} continue to hold with $3\mathcal{Q}_{k}$ replaced by $C\mathcal{Q}_{k}$ in the definitions of $\tilde S_{k}$ and $\tilde R_{k}$ for an arbitrary $C$; in our case we can take e.g.\ $C=100(a+1)$.

Alternatively, note that $\Nop_{a}$ can be seen as the usual nontangential maximal operator of aperture $a$ applied to the function $(x,s) \mapsto \sup_{t>s} F(x,s,t)$.
Hence the operator $\Nop_{a}\circ\hVop\circ\Aop$ has weak type $(1,1)$/strong type $(p,p)$ for all $a>0$ provided that this holds for some $a>0$, see e.g.\ \cite[\textsection II.2.5.1]{MR1232192}.
\end{proof}

The next lemma compares variational truncations of $\omega$-CZ kernels at nearby points.
The case $r=\infty$ of this lemma appeared in \cite[Lemma 2.3]{arXiv:1510.05789}.
\begin{lemma}
\label{lem:semicontinuity}
Let $r>1$, $x,x'\in\R^{d}$, $0<\epsilon\leq \delta\leq\infty$, and suppose $|x-x'|\leq\epsilon/2$.
Let also $K$ be an $\omega$-CZ kernel.
Then
\begin{multline*}
|\hVop\Top f(x,\epsilon,\delta)-\hVop\Top f(x',\epsilon,\delta)|
\lesssim_{d}
(\|\omega\|_{\Dini}+r' C_{K}) \sup_{\epsilon\leq t\leq \delta} A_{t}|f|(x)\\
+
C_{K} (\hVop\Aop |f|(x,\epsilon,\delta)+\hVop\Aop |f|(x',\epsilon,\delta))
\end{multline*}
\end{lemma}
Theorem~\ref{thm:nontangentional-variation} is an immediate consequence of Lemma~\ref{lem:semicontinuity}, Lemma~\ref{lem:VA-unweighted}, and the Hardy--Littlewood maximal inequality.
\begin{proof}
By the triangle inequality on $\ell^{r}$ the left-hand side of the conclusion is bounded by
\[
\sup_{\epsilon\leq t_{1}<\dots<t_{J} \leq\delta} \Big( \sum_{j=1}^{J-1} \big| \int_{t_{j}<|x-y|<t_{j+1}} K(x,y) f(y) - \int_{t_{j}<|x'-y|<t_{j+1}} K(x',y) f(y) \big|^{r} \Big)^{1/r}.
\]
For a fixed sequence $t_{1}<\dots<t_{J}$ we estimate this by
\begin{multline*}
\Big( \sum_{j=1}^{J-1} |\int_{t_{j}<|x-y|<t_{j+1}} (K(x,y)-K(x',y)) f(y)|^{r} \Big)^{1/r}\\
+
\Big( \sum_{j=1}^{J-1} |(\int_{t_{j}<|x-y|<t_{j+1}} - \int_{t_{j}<|x'-y|<t_{j+1}}) K(x',y)f(y)|^{r} \Big)^{1/r}
=: I + II.
\end{multline*}
In the first term we estimate the $\ell^{r}$ norm by the $\ell^{1}$ norm and proceed as in \cite[Lemma 2.3]{arXiv:1510.05789}:
\begin{align*}
I
&\leq
\sum_{j=1}^{J-1} \int_{t_{j}<|x-y|<t_{j+1}} |K(x,y)-K(x',y)| |f(y)|\\
&\leq
\int_{\epsilon<|x-y|<\delta} \omega\big(\frac{|x-x'|}{|x-y|}\big) \frac{|f(y)|}{|x-y|^{d}}\\
&\leq
\sum_{k=0}^{\infty}
\omega\big(\frac{\epsilon/2}{2^{k}\epsilon}\big) \int_{2^{k}\epsilon<|x-y|<\min(2^{k+1}\epsilon,\delta)} \frac{|f(y)|}{|x-y|^{d}}\\
&\lesssim_{d}
\sum_{k=0}^{\infty} \omega(2^{-k-1}) \sup_{\epsilon<t<\delta} A_{t}|f|(x)\\
&\lesssim
\|\omega\|_{\Dini} \sup_{\epsilon<t<\delta} A_{t}|f|(x).
\end{align*}

In order to estimate the second term we use an idea from \cite{arXiv:1511.05129}.
If $t_{j+1}-t_{j}\leq 2|x-x'|$, then we estimate
\[
|1_{t_{j}<|x-y|<t_{j+1}} - 1_{t_{j}<|x'-y|<t_{j+1}}|
\leq
1_{t_{j}<|x-y|<t_{j+1}} + 1_{t_{j}<|x'-y|<t_{j+1}}.
\]
Otherwise we estimate
\begin{multline*}
|1_{t_{j}<|x-y|<t_{j+1}} - 1_{t_{j}<|x'-y|<t_{j+1}}|\\
\leq
|1_{t_{j}<|x-y|} - 1_{t_{j}<|x'-y|}|
+
|1_{|x-y|<t_{j+1}} - 1_{|x'-y|<t_{j+1}}|\\
\leq
1_{t_{j}<|x-y|<t_{j}+|x-x'|}
+
1_{t_{j}<|x'-y|<t_{j}+|x-x'|}\\
+
1_{t_{j+1}-|x-x'|<|x-y|<t_{j+1}}
+
1_{t_{j+1}-|x-x'|<|x'-y|<t_{j+1}}.
\end{multline*}
Thus we may estimate $II$ by a sum of two terms of the form
\[
\Big( \sum_{j=1}^{J'-1} ( \int_{s_{j}<|x_{0}-y|<s_{j+1}} |K(x',y)| |f(y)| )^{r} \Big)^{1/r},
\]
where $x_{0}=x,x'$ and the sequence $\epsilon\leq s_{1}<\dots<s_{J'}\leq \delta$ has bounded differences: $|s_{j+1}-s_{j}|\leq 2|x-x'|$.
Using the hypothesis that $|x-x'|<\epsilon/2$ and the kernel estimate we can bound the above by a dimensional constant times
\[
C_{K} \Big( \sum_{j=1}^{J'-1} ( s_{j+1}^{-d} \int_{s_{j}<|x_{0}-y|<s_{j+1}} |f(y)| )^{r} \Big)^{1/r}.
\]
The above $\ell^{r}$ norm can be written as
\begin{multline*}
\Big( \sum_{j=1}^{J'-1} \Big( s_{j+1}^{-d} \big(\int_{|x_{0}-y|<s_{j+1}} |f(y)| - \int_{|x_{0}-y|<s_{j}} |f(y)| \big) \Big)^{r} \Big)^{1/r}\\
\leq
\Big( \sum_{j=1}^{J'-1} \big( s_{j+1}^{-d} \clapint{|x_{0}-y|<s_{j+1}} |f(y)| - s_{j}^{-d} \clapint{|x_{0}-y|<s_{j}} |f(y)|) \big)^{r} \Big)^{1/r}
+
\Big( \sum_{j=1}^{J'-1} ( (s_{j}^{-d} - s_{j+1}^{-d}) \clapint{|x_{0}-y|<s_{j}} |f(y)| )^{r} \Big)^{1/r}\\
\lesssim_{d}
\hV{A_{s}|f|(x_{0})}{\epsilon<s<\delta}
+
\sup_{\epsilon<s<\delta}A_{s}|f|(x_{0})
\Big( \sum_{j=1}^{J'-1} ( (s_{j}^{-d} - s_{j+1}^{-d})/s_{j}^{-d} )^{r} \Big)^{1/r}.
\end{multline*}

It remains to obtain a uniform bound on the last bracket.
By homogeneity we may assume $1<s_{1}<s_{2}<\dots$ and $s_{j+1}-s_{j}\leq 1$.
Then
\begin{multline*}
\Big( \sum_{j} ( (s_{j}^{-d} - s_{j+1}^{-d})/s_{j}^{-d} )^{r} \Big)^{1/r}
=
\Big( \sum_{j} (1 - (s_{j}/s_{j+1})^{d})^{r} \Big)^{1/r}\\
\leq
d \Big( \sum_{j} (1 - s_{j}/s_{j+1})^{r} \Big)^{1/r}
=
d \Big( \sum_{n\in\N} \sum_{s_{j}\in [n,n+1)} (\frac{s_{j+1} - s_{j}}{s_{j+1}})^{r} \Big)^{1/r}\\
\leq
d \Big( \sum_{n\in\N} (\sum_{s_{j}\in [n,n+1)} \frac{s_{j+1} - s_{j}}{n})^{r} \Big)^{1/r}
\leq
d \Big( \sum_{n\in\N} (\frac{2}{n})^{r} \Big)^{1/r}
\lesssim
\frac{d}{r-1}.
\qedhere
\end{multline*}
\end{proof}

The proof of Lemma~\ref{lem:semicontinuity} in fact shows that the homogeneous $r$-variation in its conclusion can be restricted to the ``short variation'' that can be controlled (for $r\geq 2$) by the uncentered square function in \cite[Theorem 1.4]{arxiv:1409.7120}.
Thus the application of L\'epingle's inequality (through the use of Lemma~\ref{lem:VA-unweighted}) to estimate the error term in the above proof is not strictly necessary (but helps us to avoid additional notation).

\section{Sparse domination}
\label{sec:sparse}
The main ingredient in the proof of Theorem~\ref{thm:sparse-domination} is the cube selection rule in Lacey's recursion lemma \cite[Lemma 4.7]{arXiv:1501.05818} and its quantitative refinement \cite[Lemma 2.8]{arXiv:1510.05789}.
It can be formulated in terms of the localized non-tangentional maximal operator as follows.

Let $F$ be a subadditive monotonic function on $\cX$.
Let $Q_{0}\in\mathcal{D}_{0}$ be a dyadic cube $\lambda : Q_{0}\to [0,\infty]$ any function defined on $Q_{0}$.
Let
\[
\sigma(y):=\inf \{ \tau>0 : F(y,\tau,\dist(y,\complement Q_{0})) \leq \lambda(y) \},
\quad y\in Q_{0},
\]
and let
\[
Y := \{y\in Q_{0}: \sigma(y)>0\}.
\]
For each $y\in Y$ choose a dyadic cube $Q_{y} \subset Q_{0}$ that contains $B(y,2\sigma(y))$ and has side length $\lesssim\sigma(y)$ (such a cube exists by definition of adjacent systems).
Let $\calQ=\calQ_{\lambda}(F,Q_{0})$ be the collection of the maximal cubes among the $Q_{y}$'s.
Then for every $y\in Y$ we have
\begin{equation}
\label{eq:hV-outside-exceptional}
F(y,\dist(y,\complement Q),\dist(y,\complement Q_{0})) \leq \lambda(y)
\end{equation}
for some $Q\in\calQ$, since this holds with $Q$ replaced by $Q_{y}$ (indeed, if the left-hand side is non-zero, then $\sigma(y)<\dist(y,\complement Q_{0})$ with strict inequality, so that by construction $\dist(y,\complement Q)>\sigma(y)$ holds also with strict inequality).
In particular, by subadditivity of $F$ we obtain
\[
\Nop_{0,Q_{0}}F
\leq
1_{Q_{0}}(\lambda + \sup_{Q\in\calQ} \Nop_{0,Q}F).
\]
\begin{lemma}
\label{lem:recursion}
Suppose that the function $\lambda(x)$ identically equals a constant $\lambda$.
Then the collection $\calQ=\calQ_{\lambda}(F,Q_{0})$ of dyadic cubes $Q\subset Q_{0}$ constructed above satisfies
\begin{equation}
\label{lem:recursion:sparse}
\sum_{Q\in \calQ} |Q| \lesssim |\{ \Nop_{Q_{0}} F > \lambda \}|
\end{equation}
and for every subadditive function $\tilde F \leq F$ we have
\begin{equation}
\label{lem:recursion:domination}
\Nop_{Q_{0}} \tilde F
\leq
1_{Q_{0}} (\lambda + \sup_{Q\in\calQ} \Nop_{Q}\tilde F).
\end{equation}
\end{lemma}

\begin{proof}
We write the left-hand side of \eqref{lem:recursion:sparse} as
\[
\sum_{\alpha} \sum_{Q\in\calQ\cap\cD^{\alpha}} |Q|
\]
and fix $\alpha$.
Since the cubes in $\calQ\cap\cD^{\alpha}$ are disjoint and each of them contains $B(y,\sigma(y))$ for some $y\in Y$ and has side length $\lesssim \sigma(y)$, the inner sum is bounded by a constant (depending on the doubling constant) times the measure of
\[
\bigcup_{y\in Y} \{ x : |x-y|<\sigma(y)\}
\subset
\{ x\in Q_{0} : \Nop_{Q_{0}}F(x) > \lambda\}.
\]

It remains to prove \eqref{lem:recursion:domination}.
If $\Nop_{Q_{0}}\tilde F(x)>\lambda$, then
\begin{multline*}
\Nop_{Q_{0}}\tilde F(x)
=
\sup_{y\in Y} \tilde F(y,|x-y|,\dist(y,\complement Q_{0}))\\
\leq
\sup_{y\in Y} \inf_{Q\in\calQ}
\Big( \tilde F(y,\dist(y,\complement Q),\dist(y,\complement Q_{0}))
+
\tilde F(y,|x-y|,\dist(y,\complement Q)) \Big)\\
\leq
\lambda + \sup_{y\in Y} \sup_{Q\in\calQ} \tilde F(y,|x-y|,\dist(y,\complement Q))
\end{multline*}
by subadditivity of $\tilde F$, the assumption $\tilde F\leq F$, and \eqref{eq:hV-outside-exceptional}.
The last summand can be non-zero only if $|x-y|<\dist(y,\complement Q)$, so that $x\in Q$, so it can be estimated by $\Nop_{Q}\tilde F(x)$.
\end{proof}

\begin{proof}[Proof of Theorem~\ref{thm:sparse-domination}]
For a cube $Q$ denote by $\cQ(Q)$ the family provided by Lemma~\ref{lem:recursion} applied $Q$ with $\lambda=c_{Q}$, so that $|Q|^{-1}\sum_{Q'\in\cQ(Q)}|Q'| \leq C_{\eqref{lem:recursion:sparse}}\epsilon$.
Therefore, in view of the doubling hypothesis, $n(Q')>n(Q)$ for all $Q'\in\cQ(Q)$ provided that $\epsilon$ is small enough.

Following the proof of \cite[Theorem 4.2]{arXiv:1501.05818}, initialize $\cP_{k_{0}} := \cup_{\alpha} \cD^{\alpha}_{k_{0}}$ and define inductively
\begin{align*}
\cP_{k}^{*} &:= \cP_{k} \cap \cup_{\alpha} \cD^{\alpha}_{k},\\
\cP_{k+1} &:= \text{maximal cubes in } (\cP_{k}\setminus\cP_{k}^{*}) \cup \bigcup_{P\in\cP_{k}^{*}} \cQ(P).
\end{align*}
The sparse collections in the conclusion of the theorem will be given by
\[
\cS^{\alpha} := \cS \cap \cD^{\alpha},
\quad
\cS := \cup_{k\geq k_{0}}\cP_{k}^{*}.
\]

Let us first verify the Carleson property for the collections $\cS^{\alpha}$.
We call the cubes $Q\in\cQ(P)$, $P\in\cP_{k}^{*}$, the \emph{$\cQ$-children} of $P$.
Note that a cube can have many $\cQ$-parents.
We claim that all $\cQ$-descendants of any cube $P$ are contained in a ball $B(z(P),C\delta^{k(P)})$, where $C$ is a constant that depends only on the quasimetric constant and $\delta$.
Indeed, if $(z_{0},z_{1},\dots)$ is a sequence of points with $\rho(z_{n},z_{n+1}) \leq C \delta^{n}$, then $\rho(z_{2^{m}n},z_{2^{m}(n+1)}) \leq A_{0}^{m} C \sigma^{n}$ with $\sigma=\delta^{2^{m}}$.
Choosing $m$ so large that $\sigma A_{0}<1$, we can estimate
\begin{multline*}
\rho(z_{0},z_{2^{m}n})
\leq
A_{0}(\rho(z_{2^{m}0},z_{2^{m}1})+A_{0}(\rho(z_{2^{m}1},z_{2^{m}2})+\dots)))\\
\leq
A_{0}^{m} C \sum_{l=0}^{\infty} (A_{0}\sigma)^{l}
\leq
\frac{A_{0}^{m}}{1-A_{0}\sigma} C,
\end{multline*}
and the claim follows.

Now let $Q,Q' \in \cS^{\alpha}$ with $Q'\subsetneq Q$, so that in particular $k(Q')>k(Q)$.
Then by construction $Q'\not\in\cP_{k(Q)}$.
On the other hand, since $Q'\in\cP_{k(Q')}$, it must have a $\cQ$-ancestor $P$ in $\cP_{k(Q)}$, and since by the above argument $Q'$ is contained in a ball of radius $C\delta^{k(P)}$ with center in $P$, the cube $P$ must in turn be contained in $B(z(Q),C\delta^{k(Q)})$ for some larger constant $C$.
Since the elements of $\cP_{k(Q)}\cap\cD^{\alpha}$ are maximal and therefore disjoint, the family $\cP_{k(Q)}$ has bounded overlap, and by the doubling property of our measure space it follows that the total measure of all possible ancestors in $\cP_{k(Q)}$ is bounded by a multiple of $|Q|$.
Moreover, if $\epsilon < 1/C_{\eqref{lem:recursion:sparse}}$, then the total mass of all $\cQ$-descendants of each $P$ is bounded by a constant times the measure of $P$.
This completes the verification of the Carleson condition.

It remains to show \eqref{eq:sparse-domination}.
Consider the family of truncations of the function $F$ given by $F_{\tau}(x,t,s) := F(x,\max(t,\tau),\max(s,\tau))$.
By induction on $K\geq k_{0}$ we obtain
\begin{equation}
\label{eq:sparse-domination:induction}
\max_{Q_{0}\in\cP_{k_{0}}}\Nop_{Q_{0}} F_{\tau}
\leq
\sum_{k=k_{0}}^{K-1}\sum_{Q\in\mathcal{P}_{k}^{*}} c_{Q} 1_{Q}
+
\max_{Q\in\cP_{K}} \Nop_{Q} F_{\tau}
\end{equation}
for each $\tau>0$.
Indeed, the base case $K=k_{0}$ holds trivially, and in the inductive step we can apply \eqref{lem:recursion:domination} and obtain
\begin{multline*}
\max_{Q\in\cP_{K}} \Nop_{Q} F_\tau
=
\max \Big\{ \max_{Q\in\cP_{K}\setminus \cP_{K}^{*}} \Nop_{Q} F_\tau,
\max_{Q\in\cP_{K}^{*}} \Nop_{Q} F_\tau \Big\}\\
\leq
\max \Big\{ \max_{Q\in\cP_{K}\setminus \cP_{K}^{*}} \Nop_{Q} F_\tau,
\max_{Q\in\cP_{K}^{*}} (c_{Q} 1_{Q} + \max_{Q'\in\cQ(Q)} \Nop_{Q'} F_\tau) \Big\}\\
\leq
\max_{Q\in\cP_{K+1}} \Nop_{Q} F_\tau
+ \sum_{Q\in\cP_{K}^{*}} c_{Q} 1_{Q}.
\end{multline*}
The second summand on the right-hand side of \eqref{eq:sparse-domination:induction} vanishes identically for each fixed $\tau>0$ and $K$ that are so large that $\delta^{K}\ll \tau$.
Thus we have obtained
\[
\max_{Q_{0}\in\cP_{k_{0}}}\Nop_{Q_{0}} F_{\tau} \leq \sum_{\alpha} \sum_{Q\in\mathcal{S}^{\alpha}} 1_{Q} c_{Q},
\]
and the left-hand side converges to $\Nop F$ pointwise as $\tau\to 0$ and $k_{0}\to-\infty$.
\end{proof}

\section{Commutators of BMO functions and CZ operators}
\label{sec:commutator}
In this section we prove a sparse domination theorem for iterated commutators of BMO functions with multilinear operators that extends \cite[Theorem 1.1]{arXiv:1604.01334}.
An $m$-linear operator $\Top$ taking an $m$-tuple $\vec f = (f_{1},\dots,f_{m})$ of functions defined on $X$ to a function defined on $\cX$ is called \emph{local} if $\Top(\vec f)(x,s,t)$ depends only on the restrictions of the functions $f_{j}$ to the ball $B(x,t)$.
The main case of interest are truncations of multilinear CZ operators.

Let $B$ be an index set and $\jmath:B\times \{0,1\} \to \{0,\dots,m\}$.
For a tuple of functions $\vec b=(b_{\beta})_{\beta\in B}$, $j\in\{0,\dots,m\}$, and an index $a\in\{0,1\}^{B}$ let $b_{a,j}:=\prod_{\beta : \jmath(\beta,a(\beta))=j} (-1)^{a(\beta)} b_{\beta}$.
The (iterated) $\jmath$-commutator of $\vec b$ with an $m$-linear operator $\Top$ is defined by
\[
[\vec b,\Top]_{\jmath}(\vec f)(x,s,t) :=
\sum_{a\in \{0,1\}^{B}} b_{a,0}(x)
\Top(\widevec{f b_{a}})(x,s,t),
\]
where $\widevec{f b_{a}}$ is the vector $(f_{1}b_{a,1},\dots,f_{m}b_{a,m})$.
Multilinear operators of this type have been studied in \cite{MR2483720}.

The next result extends \cite[Theorem 1.1]{arXiv:1604.01334}.
Note that it holds for spaces of homogeneous type; this allows one to recover a number of results in that setting, see e.g.\ \cite{arXiv:1401.2061}.
\begin{theorem}
\label{thm:bmo-czo-commutator}
For every space of homogeneous type $(X,\rho,\mu)$ and every choice of adjacent systems of dyadic cubes $\cD^{\alpha}$ there exists $0<\eta<1$ such that the following holds.
Let $1\leq r\leq\infty$ and let $\Top$ be an $m$-linear local operator such that
\begin{equation}
\label{eq:bmo-czo:endpoint-hypothesis}
C_{\Top} := \|\Nop \circ \hVop \circ \Top\|_{L^{1}\times\dots\times L^{1}\to L^{1/m,\infty}} < \infty.
\end{equation}
Let $B,\jmath,\vec b$ be as above and let $c_{\beta,Q}$ for $\beta\in B$ and $Q\in\cup_{\alpha}\cD^{\alpha}$ be arbitrary numbers.
Let also $Q_{0}$ be an initial dyadic cube and $f_{1},\dots,f_{m}\in L^{\infty}(Q_{0})$.
Then there exist $\eta$-sparse collections $\calS^{\alpha,k_{0}} \subset \calD^{\alpha}$ such that we have
\[
\Nop_{0} \hVop{} [\vec b,\Top]_{\jmath} \vec f
\lesssim
C_{\Top} \liminf_{k_{0}\to-\infty} \sum_{\alpha} \sum_{Q\in\calS^{\alpha,k_{0}}} \{ \vec b, \vec f \}_{\jmath,Q}
\]
pointwise almost everywhere, where
\[
\{ \vec b, \vec f \}_{\jmath,Q}(x) := 1_{Q}(x) \sum_{a\in\{0,1\}^{B}} |b_{a,0,Q}(x)| \prod_{j=1}^{m} \mean{| b_{a,j,Q} f_{j}|}
\]
and
\[
b_{a,j,Q}:=\prod_{\beta : \jmath(\beta,a(\beta))=j} (-1)^{a(\beta)} (b_{\beta}-c_{\beta,Q}).
\]
\end{theorem}
In absence of commutators ($B=\emptyset$) this follows directly from Theorem~\ref{thm:sparse-domination}, and in fact the centered operator $\Nop_{0}$ can be replaced by the uncentered operator $\Nop$ in the conclusion.
In presence of commutators the most interesting choice of constants is of course $c_{\beta,Q} = \mean{b_{\beta}}$.
\begin{proof}[Proof of Theorem~\ref{thm:bmo-czo-commutator}]
The only difference from Theorem~\ref{thm:sparse-domination} is that we need a suitable substitute for \eqref{lem:recursion:sparse} when
\[
F
=
\hVop{} [\vec b,\Top]_{\jmath}\vec f
\]
and
\[
\lambda(x) = \epsilon^{-1} C_B \{ \vec b, \vec f \}_{Q_{0}}(x).
\]
Note that, by multilinearity of $\Top$, the function $F$ does not change upon replacing $b_{\beta}$ by $b_{\beta}-c_{\beta,Q_{0}}$.
For each $y\in Y$ we have
\[
\lambda(y)
<
F(y,\frac12 \sigma(y),\dist(y,\complement Q_{0})).
\]
By the triangle inequality for the $\ell^{r}$ norm this implies
\[
\epsilon^{-1} C_B |b_{a,0,Q_{0}}(y)| \prod_{j=1}^{m} \mean[Q_{0}]{|b_{a,j,Q_{0}} f_{j}|}
<
|b_{a,0,Q_{0}}(y)| \hVop \Top( \widevec{f b_{a,Q_{0}}} )(y,\frac12 \sigma(y), \dist(y,\complement Q_{0}))
\]
for some $a\in\{0,1\}^{B}$.
Since this inequality is strict, the factor $|b_{a,0,Q_{0}}(y)|$ cannot be zero and can be canceled.
It follows that
\[
\bigcup_{y\in Y} B(y,\sigma_{y}/4)
\subset
\bigcup_{a \in \{0,1\}^{B}}
\Big\{ \hVop \Top(\widevec{fb_{a,Q_{0}}})
>
\epsilon^{-1} C_B \prod_{j=1}^{m} \mean[Q_{0}]{|b_{a,j,Q_{0}} f_{j}|} \Big\},
\]
and the measures of the latter sets are bounded by $\epsilon^{1/m}|Q_{0}|$ by definition of $C_B$ and locality of $\Top$.
This provides the estimate $\sum_{Q\in\calQ} |Q| \lesssim \epsilon^{1/m}|Q_{0}|$.
\end{proof}

The above domination theorem requires as an input an endpoint weak type estimate \eqref{eq:bmo-czo:endpoint-hypothesis} for $\Nop\circ\Vop\circ\Top$.
In the multilinear case such bounds are known only for $r=\infty$ (that is, for maximal truncations) and can be found in \cite{arXiv:1512.02400} (where they are stated for $X=\R^{d}$).
More precisely, the weak type estimate for $\Nop_{0}\circ\Vop[\infty]\circ\Top$ is proved in \cite[\textsection 6]{arXiv:1512.02400} and the weak type estimate for $\Nop\circ\Vop[\infty]\circ\Top$ is effectively proved in \cite[\textsection 3.1]{arXiv:1512.02400}.
The main difference from the linear case is the need to use the multilinear maximal function from \cite[Theorem 3.3]{MR2483720}.

In the linear case one can obtain the hypothesis \eqref{eq:bmo-czo:endpoint-hypothesis} with $2<r<\infty$ for a certain class of CZ operators from Theorem~\ref{thm:nontangentional-variation}.
Using the results of \cite[\textsection 4]{arXiv:1604.01334} this implies weighted estimates for variational truncations of commutators of CZ operators with BMO functions.
In fact, even unweighted estimates for variational truncations of such commutators seem to be new.

\printbibliography
\end{document}

%%% Local Variables: 
%%% mode: latex
%%% TeX-master: t
%%% ispell-local-dictionary: "american"
%%% End: 